\documentclass[article, 12pt]{amsart}
\usepackage{latexsym}
\usepackage{amsfonts}
\usepackage{amssymb}
\usepackage{mathrsfs}
\usepackage[all,pdf]{xy}
\usepackage{bbm}
\usepackage{comment}
\usepackage[mathscr]{euscript}
\usepackage{amsmath,amsthm,amssymb,mathrsfs,amsfonts,verbatim,color,leftidx}

\newtheorem{theorem}{Theorem}[section]
\newtheorem{corollary}[theorem]{Corollary}
\newtheorem{lemma}[theorem]{Lemma}
\newtheorem{prop}[theorem]{Proposition}
\newtheorem{conj}[theorem]{Conjecture}
\theoremstyle{definition}

\newtheorem{remark}[theorem]{Remark}
\numberwithin{equation}{section}

\newcommand{\RR}{\mathbb{R}}
\newcommand{\NN}{\mathbb{N}}

\newcommand{\QQ}{\mathbb{Q}}
\newcommand{\ZZ}{\mathbb{Z}}

\newcommand{\fU}{\mathfrak{U}}

\newcommand{\diag}{\mathrm{diag}}

\newcommand{\dist}{\mathrm{dist}}

\newcommand{\SL}{\mathrm{SL}}
\newcommand{\GL}{\mathrm{GL}}
\newcommand{\SO}{\mathrm{SO}}
\newcommand{\Ad}{\mathrm{Ad}}

\renewcommand{\ker}{\mathrm{Ker}}

\newcommand{\T}{\mathrm{T}}
\newcommand{\D}{\mathfrak {D}}
\newcommand{\dd}{\; \mathrm{d}}

\newcommand{\Lg}{\mathfrak{g}}
\newcommand{\Lh}{\mathfrak{g}}
\newcommand{\Lr}{\mathfrak{r}}
\newcommand{\df}{{\, \stackrel{\mathrm{def}}{=}\, }}

\newcommand {\ignore}[1]  {}

\newif\ifdraft\drafttrue
\draftfalse

\newcommand{\ggm}{G/\Gamma}

\renewcommand{\setminus}{\smallsetminus}

\begin{document}
	
	\title{Hausdorff dimension of divergent trajectories on homogeneous space}

	% Information for first author
	\author{Lifan Guan}
	% Address of record for the research reported here
	\address{Department of Mathematics, University of York, Heslington, York, YO10
5DD, United Kingdom}
\email{lifan.guan@york.ac.uk}
	% Current address
	%\curraddr{Department of Mathematics, Ohio State
	%University, Columbus, Ohio 43210}
	\email{}
		\thanks{L. G. is supported by EPSRC Programme Grant EP/J018260/1}

	\author{Ronggang Shi}
	% Address of record for the research reported here
	\address{Shanghai Center for Mathematical Sciences, Fudan University, Shanghai 200433, PR China}
	% Current address
	%\curraddr{Department of Mathematics, Ohio State
	%University, Columbus, Ohio 43210}
	\email{ronggang@fudan.edu.cn}
	% \thanks will become a 1st page footnote.
	\thanks{}
	
	%\author{}

	% General info
	\subjclass[2000]{Primary  37A17; Secondary 11K55, 37C85.}
	
	\date{}
	
	%\dedicatory{This paper is dedicated to my advisor.}
	
	\keywords{homogeneous dynamics,  divergent trajectory, Hausdorff dimension}

	\begin{abstract}
	For one parameter subgroup action on a  finite volume homogeneous space,  we
	consider the set of points admitting divergent on average trajectories. We
	show that  the Hausdorff dimension of this set
	 is strictly less than the manifold
	dimension of the homogeneous space. As a corollary we know  that the Hausdorff dimension of the  set of
	points admitting divergent trajectories is not full, which proves a conjecture of Y. Cheung \cite{c11}.
	\end{abstract}

	\maketitle
	
	\markright{}
	
\section{Introduction}

Let $G$ be a connected Lie group, $\Gamma$ be a lattice of $G$\footnote{A discrete subgroup $\Gamma<G$ is called a lattice if there exists a finite $G$-invariant measure on the homogeneous space $\ggm$.} and $F=\{f_t: t\in \RR \}$ be a one parameter  subgroup of
$G$. The   action of $F$ on the homogeneous space  $G/\Gamma$ by left translation defines a flow. In this paper we consider the dynamics of the  semiflow given  by the action of  $F^+
\df \{ f_t: t\ge 0 \}$.
For $x\in G/\Gamma$ we say
 the trajectory  $F^+x\df \{f_t x: t\ge 0 \}$ is \emph{divergent} if
  $f_t x$ leaves  any fixed compact  subset of $G/\Gamma$ provided $t$ is   sufficiently large. We say $F^+x$ is  \emph{divergent on average}  if
for any characteristic function $\mathbbm {1}_K$ of a compact subset $K$ of $G/\Gamma $ one has
\[
\lim\limits_{T\to \infty}\frac{1}{T}\int_0^T \mathbbm{1}_K (f_t x)\dd t= 0.
\]
Clearly, if the trajectory $F^+x$ is divergent, then it is  divergent on average.
The aim of this paper is to understand the set of divergent     points
$$
\D'(F^+, G/\Gamma)\df\{x\in \ggm: F^+x \text{ is divergent}\},
$$
and the set of divergent on average   points
$$\D(F^+, G/\Gamma)\df\{x\in \ggm: F^+x\text{ is divergent on average}\}, $$
in terms of their Hausdorff dimensions. Here the Hausdorff dimension is defined by attaching $\ggm$ with a Riemannian metric. It is well-known that different choices of Riemannian metrics will not affect the Hausdorff dimension of subsets of $\ggm$. Indeed, specific Riemannian metric will be used later for the  sake of convenience.

According to the work of Margulis \cite{Ma} and Dani \cite{d84}\cite{d86}, it is
well-known that if $F$ is   $\Ad$-unipotent then the space $G/\Gamma$ admits  no divergent  on average   trajectories of   $F^+$. In other words, the set $\D(F^+, G/\Gamma)$, hence the set $\D'(F^+, G/\Gamma)$, is empty.
On the other hand, the set $\D(F^+, G/\Gamma)$ can be complicated when $F$ is $\Ad$-diagonalizable.
 For example, it was proved by  Cheung in \cite{c11} that
the Hausdorff dimension of  $\D'(F^+, \SL_3(\mathbb R)/\SL_3(\mathbb Z))$ with $F=\{\diag(e^t, e^t, e^{-2t}): t\in  \RR\}$
is equal to $7\frac{1}{3}$.  Based on his results, Cheung raised the following conjecture in \cite{c11}.
\begin{conj}
Let $\Gamma$ be a lattice  of a connected Lie group  $G$  and
	let $F=\{f_t: t\in  \RR \}$ be a one parameter subgroup of $G$. Then the Hausdorff  dimension of
	$\D'(F^+, G/\Gamma)$
	is strictly less  than the manifold dimension of $G/\Gamma$.
\end{conj}
The conjecture is known to be true in the following cases where $G$ is a semisimple
Lie group without compact factors and  $F$ is $\Ad$-diagonalizable:
\begin{enumerate}
  \item $G$ is of rank one \cite{Dani}.
  \item $G=\prod_{i=1}^{n}\SO(n,1)$, $\Gamma=\prod_{i=1}^{n}\Gamma_i$ with each $\Gamma_i$ lattice in $\SO(n,1)$ and $F< G$ the is  diagonal embedding of any one parameter  real split torus $A$ of $\SO(n,1)$ \cite{c07}\cite{yang}.
  \item $\ggm=\SL_{m+n}(\RR)/\SL_{m+n}(\ZZ)$ and $F=F_{n,m}=\{\diag(e^{nt},\ldots,  e^{nt}, e^{-mt}, \ldots, e^{-mt}): t\in \RR \}$ with $m,n\ge 1$ \cite{c11}\cite{cc16}\cite{KKLM}.
\end{enumerate}
Indeed, for all the cases listed above, the Hausdorff dimension of the corresponding $\D'(F^+, G/\Gamma)$ have been determined.
%it was conjectured by Cheung \cite{c11} that
%the Hausdorff dimension of $\D'(F^+, G/\Gamma)$ is strictly   less than  the manifold dimension of $G/\Gamma$ for %any $F$, $G$ and $\Gamma$.

There are evidences that a stronger version of this conjecture is true.
It was proved
 by
Einsiedler-Kadyrov in \cite{ek12} that the Hausdorff dimension of   $\D(F^+, \SL_3(\mathbb R)/\SL_3(\mathbb Z))$ is at most $7\frac{1}{3}$ when $F=F_{1,2}$ as in (3). Using the contraction property of the height function introduced in \cite{emm98}, it was proved
by Kadyrov, Kleinbock, Lindenstrauss  and  Margulis in \cite{KKLM} that for any $m,n\ge 1$, the Hausdorff dimension of
$\D(F^+, \SL_{m+n}(\mathbb R)/\SL_{m+n}(\mathbb Z))$  is at most  $\dim G- \frac{mn}{m+n}$ when $F=F_{n,m}$ as in (3).
See also \cite{elmv}\cite{k}\cite{kp}\cite{lsst}\cite{dfsu}\cite{weiss} for related results.
%Recently, Cheung and Chevallier  extends \cite{c11} and proved that  for $n\ge 3$ the Hausdorff dimension of
%$\D_0(F^+, \SL_n(\mathbb R)/\SL_n(\mathbb Z))$ where  $F^+=\{\diag(e^t,\ldots,  e^t, e^{-(n-1)t}): t\ge 0\}$
%is equal to $\dim G-\frac{n-1}{n}$.
%In this paper   we  give nontrivial  upper bound of the Hausdorff dimension of $\D(F^+, G/\Gamma)$ for arbitrary
%$F^+$ and $G/\Gamma$.

Now we state the main result of this paper, from which Cheung's conjecture follows.
\begin{theorem}\label{t-main}
	Let $\Gamma$ be a lattice  of a connected Lie group  $G$  and
	let $F=\{f_t: t\in  \RR \}$ be a one parameter subgroup of $G$. Then the Hausdorff  dimension of
	$\D(F^+, G/\Gamma)$
	is strictly less  than the manifold dimension of $G/\Gamma$.
\end{theorem}

We will reduce the proof of  Theorem \ref{t-main}  to
the special case where $G$ is a semisimple linear group.
 Recall that a  connected  semisimple  Lie group $G$ contained in $\SL_k(\RR)$ has a natural structure of
 real algebraic group. So terminologies of  algebraic groups  have natural meanings
 for $G$ and  are independent of the embeddings of $G$ into $\SL_k(\RR)$. In particular,
 the one parameter group $F$ has the following real Jordan decomposition
which  is a special case  of \cite[Theorem 4.4]{borel}.
 \begin{lemma}\label{l-RJD}
 	Let $G\le \SL_k(\RR)$ be a connected  semisimple Lie group. For any one parameter subgroup $F=\{f_t: t\in \RR\}$, there are uniquely determined  one parameter subgroups $K_F=\{k_t:t\in \RR\}$, $A_F=\{a_t:t\in \RR\}$ and $U_F=\{u_t:t\in \RR\}$ with the following properties:
 	\begin{itemize}
 		\item $f_t=k_ta_tu_t$.
 		\item $K_F$ is bounded, $A_F$ is $\RR$-diagonalizable and  $U_F$ is unipotent.
 		\item All the  elements of  $K_F$, $A_F$ and $U_F$ commute with each other.
 	\end{itemize}
 \end{lemma}

 The subgroups $K_F, A_F$ and $ U_F$ are called compact, diagonal and unipotent parts of $F$, respectively.
% A one parameter subgroup $F$ is said to be  {\em quasiunipotent} if its diagonal part $A_F$ is trivial.
% Otherwise,   $F$ is said to be {\em non-quasiunipotent}.
 In \S\ref{sec;2} we will reduce the proof of Theorem \ref{t-main} to its  following special case which contains the main
 unknown situations.
% Note that a one parameter semigroup $F^+$ uniquely determines a one parameter group $F$ and vice versa, so  the same terminologies apply  for $F^+$.
 \begin{theorem}\label{thm;general}
 	Let $ G \le \SL_k(\RR)$ be a connected center-free semisimple Lie group without compact factors.
 	Let
 	$F=\{f_t: t\in  \RR\}$ be a one parameter subgroup of $G$
 	% where $Z_G(H)=\{g\in G: gh=hg\  \forall h\in H  \}$
 	such that the compact part
 	$K_F$ is trivial but the diagonal part
 	$A_F$  is nontrivial.
 We assume the followings hold:
 \begin{itemize}
 	\item  $G=\prod_{i=1}^m G_i$ is a direct product of connected normal subgroups
 	$G_i$.
 	\item $\Gamma =\prod_{i=1}^m \Gamma_i$ where each $\Gamma_i$ is a nonuniform irreducible
 	lattice of $G_i$.
 	\item  The group  $A_F$ has nontrivial projection to each  $G_i$.
	 \end{itemize}
  %	Suppose    $\Gamma $ is an irreducible nonuniform   lattice of $G$ and $\Gamma$ is
% 	commensurable with $\SL_k(\ZZ)$ if the rank of $G$ is bigger than one.
 	% such that $\Gamma = G\cap \SL_k(\RR)$ if the rank of $G$ is not one.	
 		 Then
 		 %for any $z\in G/\Gamma $ and any connected  closed subgroup $Q$ with $H\le Q\le
 		 %Z_{G}(H) $
 		 the Hausdorff dimension of
 		$
 	 \D(F^+, G/\Gamma)
 		$
 		is strictly less than the manifold dimension of $G/\Gamma$.
\end{theorem}

%Here the metric on $Q$ is given by the restriction of $\dist(\cdot, \cdot)$ on $G$.
%We are not trying state the most   optimal version
%of  Theorem \ref{thm;general},  but we would rather state a version   that fits the frame work of the proof of Theorem \ref{t-main}.

%In the special case where
%$G$ is  semisimple and has no compact factors, Theorem \ref{thm;general} implies Theorem \ref{t-main}.
%In \S \ref{sec;2}, we will reduce both Theorems \ref{t-main} and \ref{thm;general} to  this special case. Then we  further reduce Theorem \ref{thm;general} to Theorem \ref{thm;general}where $G $ is a semisimple  linear group (i.e.~ $G \le  \SL_k(\RR)$)   and $\Gamma $ is an irreducible lattice.
%The rest of  the paper is devoted to  the proof
 %of Proposition \ref{l-higher-rank}.

 The proof of Theorem \ref{thm;general} is  from \S \ref{sec;3}   till the end of the paper.
 Indeed, the upper bound of the Hausdorff dimension in the setting of Theorem \ref{thm;general} can  be explicitly calculated and we will make this point clear during the proof.

 Our main tool will be the  Eskin-Margulis height function (abbreviated as EM height function) introduced in \cite{em}. If $F$ is diagonalizable, i.e.~$F=A_F$, Theorem \ref{t-main} can be established using the strategy developed in \cite{KKLM} and the contraction property of the
 proved in \cite{s}.
 But when $F$  has nontrivial  unipotent parts, i.e.~$U_F$ is nontrivial, essential new ideas are needed. The following example contains the main difficulties we need to handle  in the proof of Theorem \ref{thm;general}:  $ G=\SL_4(\RR)\times \SL_4(\RR), \Gamma =\SL_4(\ZZ[\sqrt 2]) $ which embeds in $G$ diagonally via Galois conjugates, and
 \[
f_t =
 \left(
 \begin{array}{cccc}
 e^t & 0 & 0 & 0\\
 0 & e^{-t} & 0  & 0\\
 0 & 0 &  1 & t \\
 0 & 0 & 0 & 1
 \end{array}
 \right)\times
 \left(
 \begin{array}{cccc}
 1 & t & 0 & 0\\
 0 & 1 & 0  & 0\\
 0 & 0 &  1 & 0 \\
 0 & 0 & 0 & 1
 \end{array}
 \right).
 \]
 There are two main difficulties.  One is caused by  the unipotent part of $f_t$ in the first
 $\SL_4(\RR)$ factor, and the other is caused by the unipotent part of $f_t$ in the second
 $\SL_4(\RR)$ factor. To overcome these difficulties, we will
prove a uniform  contraction property for a family of one parameter subgroups in \S \ref{sec;3} and \S \ref{sec;height} with respect to the EM height function. Then the last two sections are devoted to the proof of Theorem \ref{thm;general}.

\section{Proof of Theorem \ref{t-main}}\label{sec;2}
In this section we prove Theorems  \ref{t-main} assuming Theorem   \ref{thm;general}. Let $G, \Gamma, F$ be as in Theorem \ref{t-main}. We choose and fix a Euclidean norm $\|\cdot\|$  on the Lie algebra $\Lg$ of $G$, which induces a right invariant
Riemannian metric $\dist(\cdot, \cdot)$ on $G$. Moreover, this metric naturally induce a metric on $G/\Gamma$, also denoted by  $``\dist"$, as follows:
\[
\dist(g\Gamma, h\Gamma )=\inf_{\gamma\in \Gamma}\dist(g\gamma, h)\quad \mbox{where }g, h\in G .
\]

 Let  $\Lr$ be the maximal amenable ideal  of the Lie algebra $\Lg$ of $G$, i.e.~the largest ideal whose analytic subgroup is amenable.
 The adjoint action of $G$ on $\mathfrak s=\Lg/\Lr$ defines  a homomorphism $\pi:G\mapsto \mathrm{Aut}(\mathfrak s)$.
 Let $S$ be  the connected component of $\mathrm{Aut}(\mathfrak s)$.
 It follows from the Levi decomposition of $G$ that   $\pi(G)=S$ and $S$ is a
 center-free  semisimple Lie group  without compact factors.
  It is known  that $\Gamma\cap \ker(\pi)$ is a cocompact lattice in $\ker(\pi)$ and $\pi(\Gamma)$ is a lattice in $S$, see e.g.~\cite[Lemma 6.1]{BQ}. Therefore,
 {the induced  map $\overline {\pi}: G/\Gamma \mapsto S/\pi(\Gamma)$ is proper}, and
 consequently
 \begin{align}\label{eq;equal}
 	\overline{\pi}(\D(F^+, G/\Gamma))=\D(\pi(F^+), S/\pi(\Gamma)).
 \end{align}

 Let $\varphi: \mathfrak s\to \mathfrak g$ be an embedding of Lie algebras such that $\mathrm{d}\pi\circ \varphi$ is the identity map.
 It follows from (\ref{eq;equal}) that  for any $x\in G/\Gamma$ and any $v\in  \mathfrak s$
 \[
 \exp(v)\overline{\pi}(x)\in \D(\pi(F^+), S/\pi(\Gamma))
 \]
 if and only if
 $$\exp(\varphi(v))\exp (v')x\in \D(F^+, G/\Gamma) \mbox{\quad for all }v'\in \Lr.$$
 By Marstrand's product theorem, the Hausdorff dimension of the product   of two sets of Euclidean spaces
 is bounded from above by the sum of the   Hausdorff dimension of one set  and the packing dimension of the other, e.g.~\cite[Theorem 3.2.1]{bp}. So to prove Theorem \ref{t-main} it suffices to give a nontrivial  upper bound of the Hausdorff dimension of $\D(\pi(F^+), S/\pi(\Gamma)) $.

 % In the setting of Theorem \ref{thm;general} where $H$ is a connected semisimple subgroup without compact factors, the restriction of $\pi $ to $H$ has  discrete kernel. Also,  for any
 %$z\in G/\Gamma$ and $h\in H$,
 %\[
 %h z\in \D(F^+, G/\Gamma )\Leftrightarrow  \pi(h) \overline{\pi}(z)\in\D(\pi(F^+), S/\pi(\Gamma)).  \]
% So to prove Theorem \ref{thm;general} it suffices to give a nontrivial upper bound of
% $$
% \{s \in \pi (H)| s \overline{z}  \in \D(\pi(F^+), S/\pi(\Gamma))   \}
% $$
% for any $\overline{z}\in  S/\pi(\Gamma)$.

 We summarize what we have obtained as follows.
 \begin{lemma}
 	\label{prop;more}
 	Theorem  \ref{t-main} is  equivalent to its special case where
 	$G$ is a  center-free  semisimple Lie group without compact factors.
 \end{lemma}

%The following proposition  is a special case of
 %Theorem \ref{thm;general} whose proof started from \S \ref{sec;3} and will last till the end of the paper.

 %\begin{prop}\label{l-higher-rank} Let $G, H, \Gamma , F^+$ be as in Theorem \ref{thm;general}. We  moreover assume the followings:\begin{itemize}
 %\item  $G\le  \SL_k(\RR)$ be a connected  semisimple Lie group without compact factors;
% \item $H$ is a simple Lie group;
% \item $\Gamma$ is an irreducible lattice of $G$;
% \item Either the rank of  $G$  is one or $\Gamma =G\cap \SL_k(\ZZ)$;
% \item  $F^+$ is a non-quasiunipotent one parameter semigroup   with trivial compact part $K_{F^+}$.
% \end{itemize} Then for any $z\in G/\Gamma $ the Hausdorff dimension of
% $	\{ h\in H :   h z \in \D(F^+, G/\Gamma) \}$ is strictly less than the
 %manifold dimension of $H$.
 %\end{prop}

\begin{proof}[Proof of Theorem \ref{t-main}]
According to Lemma \ref{prop;more}, it suffices to prove the theorem under the
additional assumption that
 $G$ is a center-free semisimple Lie group  without compact factors.
 Under this assumption, the adjoint representation $\Ad: G\to \SL(\Lg) $ is a closed embedding.
  According to the real Jordan decomposition in
 Lemma \ref{l-RJD}, the  compact part  $K_{F}$ does not affect the divergence on average  property of the trajectories.
 So we
 assume without loss of generality that
 $K_{F}$  is trivial.

  There exist finitely many connected  semisimple subgroups $G_i$
such that $G= \prod_i G_i$ and $\Gamma_i \df \Gamma \cap G_i$  is an irreducible lattice of $G_i$ for each $i$.
It follows that  $\prod_i \Gamma_i$ is a finite index subgroup of $\Gamma$ and the natural quotient  map  $G/\prod_i \Gamma_i \rightarrow \ggm$ is proper.
So we  assume  moreover   that $\Gamma= \prod_i \Gamma_i$.

Denote by $\pi_j $  the projection of $G$  to $G/ G_j=\prod_{i\neq j}G_i$ and
denote by $\overline{\pi}_j$ the induced map from  $G/\Gamma$ to
$\pi_j(G)/\pi_j(\Gamma)=\prod_{i\neq j} G_i/\Gamma_i$. Here if $G=G_j$ we interpret $\prod_{i\neq j}G_i$ as a trivial group and $\prod_{i\neq j} G_i/\Gamma_i$ as a singe point set.
If $G_j/\Gamma_j$ is compact or the projection of  $A_F$ to $G_j$ contains only the neutral element, then
 $$\D(F^+, G/\Gamma) =   \overline{\pi}_j^{-1}\Big(\D\big(\pi_j(F^+), \prod_{i\neq j}G_i/\Gamma_i\big)\Big).
$$
So either $\D(F^+, G/\Gamma)$ is an empty set or we finally can reduce the problem to the setting of Theorem \ref{thm;general} where
each  $\Gamma_i$ is a nonuniform lattice and the projection of $A_F$
to each $G_i $ is nontrivial. This completes the proof.

\begin{comment}

 If ${F}$ is unipotent, then the set $\D(F^+, G/\Gamma)$ is empty and Theorem \ref{thm;general} holds trivially.
 If   $A_{F}$ is nontrivial but $K_{F} $ is trivial, Then we claim that we can apply
Theorem \ref{thm;general}  to get  the conclusion.

 If the rank of $G$
 is one,
 then we can choose the embedding $ G\to \SL(\Lg)$ so that all the
 assumptions of Theorem \ref{thm;general} are satisfied.
   If the rank of $G$ is bigger than two, then by the Margulis' arithmeticity theorem, there exists an embedding $\rho: G\to \SL_k(\RR)$ such that $\rho(\Gamma)$ is commensurable with $G\cap \SL_k(\ZZ)$.    This completes the proof of the claim and hence the theorem.

\end{comment}
\end{proof}

%\begin{proof}[Proof of Theorem \ref{t-main}]
%It follows from Theorem \ref{thm;general} and Proposition \ref{prop;more}.

%\end{proof}

%Note that, for convenience, we take a one parameter subsemigroup $F^+=\{f_t: t\ge 0\}$  and the corresponding %one parameter subgroup $F=\{f_t: t\in \RR\}$ with a chosen direction, i.e. a chosen $f_1$ as the same thing.

\section{Preliminary on linear representations}\label{sec;3}
From this section, we start the proof of Theorem \ref{thm;general}. At the beginning of each section we will set up some notation that will be used later.
Let    $G$ and $ F$ be as in Theorem \ref{thm;general}.
  Let $A_F=\{a_t:t\in \RR \}$ and $U_F=\{u_t: t\in \RR \}$ be the diagonal and unipotent parts of $F$.   Let $H$ be the unique connected normal subgroup of $G$ such that
  $A_F\le H $ and the projection of $A_F$ to each simple factor of $H$ is nontrivial.
  Since $A_F$ is nontrivial and $G$ is center-free, $H$ is (nontrivial) product of some simple factors of $G$.
  Hence $H$ is a semisimple Lie group without compact factors. Let $S$ be the product of simple factors of $G$ not contained in $H$. Then $S$ is also semisimple normal subgroup of $G$ that commutes with $H$. Moreover, $G=HS$ and $H\cap S=1_G$, where $1_G$ is the neutral element of $G$.
   %Let $S$ be the connected normal subgroup of $G$ complementary to
  %$H$, i.e.~the intersection $H\cap S$ is discrete, $S$ commutes with $H$  and $G=HS$.

In this section we prove  a couple of auxiliary results for a finite dimensional linear representation $\rho: G\rightarrow \GL(V)$
on a  (nonzero real) normed vector space $V$. These results will be used in the next section to prove the uniform contracting property of EM height function. We will use $\|\cdot\|$ to denote the norm on $V$.

 For $\lambda\in \RR$,
 we  denote the $\lambda$-Lyapunov subspace of $A_F$ by
 $$V^{\lambda}=\{v\in V: \rho (a_t) v=e^{\lambda t} v \}.$$
 Recall that if $V^\lambda\neq \{0 \}$, then $\lambda$ is a called an Lyapunov exponent
 of  $(\rho, V)$.
 Since $U_F$ commutes with $A_F$, every Lyapunov  subspace $V^{\lambda}$ is $U_F$-invariant.
As $A_F$ is $\RR$-diagonalizable,
the space  $V$ can be decomposed as $V^+\oplus V^0\oplus V^-$ where
$$V^+=\oplus_{\lambda>0}V^{\lambda}\quad\mbox{and }\quad  V^-=\oplus_{\lambda <0}V^{\lambda}.$$
%Suppose $V$ has the decomposition \ref{eq;complicate}, let

Now we consider  the adjoint representation of $G$ on the  Lie algebra $\Lh$ of $G$.
%we write
%\begin{align}\label{eq;adjoint}
%\lambda^-=\lambda^-_\Lg\quad \mbox{and }\quad\lambda^+=\lambda_\Lg.\end{align}
It is easily checked that
$\Lh^+,  \Lh^-$ and $ \Lh^c\df \Lh^0$ are subalgebras of $\Lh$.
The connected subgroup  $G^+$ (resp.~$ G^-$)   with Lie algebras
$\Lh^+$ (resp.~$\Lh^-$)  is called unstable (resp.~stable) horospherical subgroup of $a_1$.
We denote    the connected component of the centralizer of $a_1$ in $G$ by $G^c$ whose  Lie algebra  is $\Lh^c$.  Let $d, d^c, d^-$ be the manifold dimensions of $G^+$,
 $G^c$ and $G^-$,  respectively.
It follows from the   nontriviality of $A_F$ that $d>0$.

For $r\ge 0$ we let $B_r^{G}=\{h\in G: \dist (h, 1_G)<r  \}$, $B_r^{\pm }=\{h\in G^{\pm}: \dist (h, 1_G)<r  \}$ and  $B_r^c=\{h\in G^c: \dist (h, 1_G)<r  \}$.
By rescaling the Riemannian metric if necessary, we may assume  that:
\begin{enumerate}
  \item the product map $B_1^-\times B_1^c\times B_1^+\to G$ is a diffeomorhism onto its image,
  \item and the logarithm map is well-defined on $B_1^{G}$ and is a diffeomorphism onto its image.
\end{enumerate}
 According to (1), it is safe to identity
the product $B_1^-\times B_1^c\times B_1^+$ with $B_1^-B_1^cB_1^+$ and we will mainly  use the later notation for sake of convenience.  The same statement as (2) also holds for $B_1^{\pm }$ and $B_1^c$.

 We fix a Haar measure  $\mu$   on $G^+$ normalized with $\mu(B_1^+)=1$.
 Since the metric $``\dist"$ is right invariant, any open ball of radius $r$ in $G^+$ has the form
 $B_r^+ h$ $(h\in G^+)$ and there exits $C_0\ge 1$ such that
 \begin{align}\label{eq;c+}
 C_0^{-1}r^d\le	\mu( B_r^+ h)=\mu{(B_r^+)}\le C_0 r^d\quad \mbox{for all } 0\le r\le 1  .
 \end{align}
  For $g, h\in G$ we let $g^h= h^{-1}g h$.  For  $z\in G^c$,
 let
 $F  _z  = \{f_t  ^z  : t\in \RR \}$ and $F  _z  ^+= \{f_t  ^z  : t \ge 0 \}$.
Note that
 $f_t  ^z  = a_t u_t  ^z  $
  and
 \begin{align}\label{eq;compact}
 \{ u_1  ^z  : z\in B_1^c \} \quad \mbox{is relatively compact. }
 \end{align}

% Let  $\|\cdot\|$ be a norm  on $V$.

\begin{lemma}\label{l-v+-}
	Let $\rho: G\to \GL(V)$ be a representation on a finite dimensional normed vector space $V$.
%	Suppose that  $\lambda\in \RR$ such that  $V^\lambda\neq \{0\}$. 	
Let $\lambda$ be a Lyapunov exponent of $(\rho, V)$.
	For any $0<\delta<1$, there exists $T_\delta >0$ such that, for all $t\ge T_\delta, z\in B_1^c$ and unit vector  $v\in V^{\lambda}$ we have
	\begin{equation}\label{ine-v+-}
	 e^{(1-\delta)\lambda t}\le \|\rho(f_{t}  ^z  )v\|
 \le e^{(1+\delta)\lambda t}.
	\end{equation}
\end{lemma}
\begin{proof}
%	Since different norms on $V$ are equivalent, we assume without loss of generality that $\|\cdot\|$ is a Euclidean norm and   different eigenspaces of $\rho(a_1)$ are orthogonal to each other.
	 	 For all $v\in V^\lambda$ with $ \|v\|=1$ we have
%	\|\rho(a_t)v\|\ge e^{\lambda t} \quad \mbox{and }\quad
$	\|\rho(a_t)v\|=  e^{\lambda t}$.
	On the other hand, in view of (\ref{eq;compact}), there exists $C>0$
	and $n\in \NN$\footnote{Here  $\NN=\{1, 2, 3, \ldots \}$.} such that
	\[
	\|\rho (u_{ t}  ^z  ) \|\le C (|t|+1)^n
	\]
	for all $z \in B^c_1$ and $t\in \RR$.
	Therefore, for any unit vector $v\in V^\lambda, z\in B_1^c$ and sufficiently large  $t$,
	\begin{align*}
	\|\rho(f_{ t}  ^z  )v\|&\ge  \|\rho(u_{ -t}  ^z  )\|^{-1} \|\rho(a_t)v\|\ge  C^{-1}(|t|+1)^{-n}e^{\lambda t}\ge e^{(1-\delta)\lambda t}, \\
	\|\rho(f_{ t}  ^z  )v\|&\le  \|\rho(u_{ t}  ^z  )\| \|\rho(a_t)v\|\le  C(|t|+1)^ne^{\lambda t}\le e^{(1+\delta)\lambda t}.
	\end{align*}

\end{proof}

From now on till the end of this section, we assume that
$\rho: G\to \GL(V)$ is a representation on a finite dimensional normed vector space $V$ which has no nonzero $H$-invariant vectors. As any two norms on $V$ are equivalent, we also assume  that the norm is Euclidean without loss of generality.

Recall that a nonzero $H$-invariant subspace $V'$ of $V$  is said to be $H$-irreducible if
$V'$ contains no $H$-invariant subspaces  besides $\{0 \}$ and itself.
The complete reducibility of representations of  $H$  implies that
there exists a unique  decomposition (called $H$-isotropic decomposition)
\begin{align}\label{eq;complicate}
V=V_1 \oplus \cdots \oplus V_m
\end{align}
such that  irreducible  sub-representations of  $H$ in
the same  $V_i$ are isomorphic but irreducible  sub-representations in different $V_i$ are non-isomorphic.  Since $S$ commutes with $H$, each $V_i$ is $S$-invariant, and hence $G $-invariant.
%  We say $V$ has only one nontrivial  $H$-isotropic type  if $m$ in (\ref{eq;complicate}) is equal to one and $V$ has no nonzero $H$-invariant  vectors.
Each $V_i$ is called an $H$-isotropic subspace of $V$.

Let $\lambda_i$ be the top Lyapunov exponent of $A_F$ in $(\rho, V_i)$, i.e.,
\begin{align*}
%\label{eq;top}
\lambda_i= \max\{ \lambda\in \mathbb R:  V_i^{\lambda}\neq \{ 0\} \}.
\end{align*}
Since  the projection of $A_F$ to each simple factor of $H$ is nontrivial, every
$\lambda_i$ is positive.
Let   $\lambda$ be  the minimum of top Lyapunov exponents in each $V_i$, i.e.
\begin{align}
\label{eq;top}
\lambda= \min\{ \lambda_i: 1\le i\le m \}>0.
\end{align}
Let $\pi_i:
V_i\to V^{\lambda_i}_i$ be the $A_F$-equivariant projection.

\begin{lemma}\label{lem;sharp}
%	Let $\rho: G\to \GL(V)$ be a representation on a finite dimensional normed vector space $V$. We assume $V$ has no nonzero $H$-invariant vectors and  it has only one  $H$-isotropic type.
%Let $\lambda$ 	be the top Lyapunov exponent of $A_F$ in $(\rho, V)$ and
For all $v\in V_i\setminus \{0\}$,
	 the map
	 \begin{align}\label{eq;varphi}
	 \varphi_v: G^+ \mapsto \RR \quad \mbox{where}\quad  \varphi_v(h)=\|\pi_i(\rho(h)v)\|^2
	 \end{align}
	 is not  identically  zero.
\end{lemma}	
\begin{proof}
	Suppose $\varphi_v$ is identically zero.
	Then $\rho(G^+) v\subset V'_i$ where
	 $V'_i\subset V_i$ is the $A_F$-invariant complimentary subspace of $V^{\lambda_i}_i$.
	 This implies that $\rho(G^-G^cG^+) v\subset V'_i$. Since $G^-G^cG^+$ contains an
	 open dense  subset of $G$, see e.g.~\cite[Proposition 2.7]{mt94}, we moreover have that $\rho(G)v\subset V'_i$. This is impossible since
	 the intersection of   $V^{\lambda_i}_i$ with each $H$-invariant subspace of $V_i$ is  nonzero.
	 This contradiction completes the proof.
\end{proof}

\begin{lemma}\label{l-c-alpha}
%	Let $\rho: G\to \GL(V)$ be a representation on a finite dimensional normed vector space $V$. We assume $V$ has no nonzero $H$-invariant vectors and  it has only one  $H$-isotropic type. Let $\lambda$ be the top Lyapunov exponent of $A_F$ in $(\rho, V)$ and $\pi: 	V\to V^\lambda$ be the $A_F$-equivariant projection.
	For all  $v\in V_i\setminus \{0\}$ and $r\ge 0$, let
	$$E(v,r)=\{h\in B_1^+ : \|\pi_i(\rho(h)v)\|\le r \}.$$
	Then there exists $\theta_i>0$ such that
	\begin{equation}\label{constant-c-alpha}
	C_i\df \sup_{\|v\|=1,v\in V_i}r^{-\theta_i}\mu(E(v,r))<\infty.
	\end{equation}
	In particular,
%	\begin{equation}\label{mu-ev-0}
$	\mu(E(v,0))=0$.
% \text{ for all } v\in V\setminus \{0\}.
%	\end{equation}
\end{lemma}

\begin{proof}
	
	Since $G^+$ is a unipotent group, it is simply connected and by \cite[Theorem 1.2.10 (a)]{cg} there is an isomorphism of affine varieties
	$\RR^{d}\to G^+$ such that the Lebesgue measure of $\RR^d$ corresponds to the Haar measure $\mu$.  During the proof, we will  identify the group $G^+$ with $\RR^{d}$ for convenience.

	By Lemma \ref{lem;sharp}, for every nonzero $v\in V_i$ the map $\varphi_v$ in (\ref{eq;varphi}) is a nonzero polynomial map. So $\varphi_v|_{B_1^+}$
	is nonzero.
	Note that the degrees of $\varphi_v$ ($v\in V_i$) are uniformly bounded from above.
	 Therefore,  the $(C, \alpha)$-good property of
	polynomials in \cite[\S 3]{bkm} implies that there exist positive constants $C$ and $ \alpha$ such that
	\begin{align}\label{eq;1}
	\mu(E(v, r))\le C \left (\frac{r^2}{ \sup_{h\in B_1^+}\varphi_v(h)}\right)^\alpha
	\end{align}
for all nonzero $v\in V_i$.
	Since the set of  unit vectors of
	$V_i$ is compact,
	\begin{align}\label{eq;2}
	\inf_{\|v\|=1, v\in V_i} \sup_{h\in B_1^+}\varphi_v(h)>0.
	\end{align}
	So (\ref{constant-c-alpha}) follows from (\ref{eq;1}) and (\ref{eq;2}) by taking
	$\theta_i=2\alpha$. 	
\end{proof}

\begin{remark}\label{rem-3}
 According to  \cite[Lemma 3.2]{bkm} we have
 $\alpha=\frac{1}{dl}$ where $d$ is the manifold
dimension of $G^+$ and $l$ is a uniform upper bound of the  degree of $\varphi_v \ (v\in V_i)$.  So the constant $\theta_i$ can be calculated explicitly.
%Consequently, the contraction rate defined in \S \ref{sec;height} can be calculated explicitly.
\end{remark}

\begin{lemma}\label{l-contraction}
%	Let $\rho: G\to \GL(V)$ be a   representation  on a finite dimensional normed vector space $V$ which has no nonzero $H$-invariant vectors. Let $\lambda$ be the minimum of the top Lyapunov exponents in $H$-isotropic subspaces.
     Let
	 $\theta_0= \min_{ 1\le i\le m} \theta_i$ where $\theta_i>0$  so that   Lemma \ref{l-c-alpha} holds and let $\lambda$ be as in (\ref{eq;top}).   Then for any    $0<\delta<\theta < \theta_0$, there exists $T_{\theta, \delta}>0$ such that for all $t\ge T_{\theta,   \delta}$, $z\in B_1^c$ and $v\in V$ with $\|v\|=1$, we have
	\begin{equation}\label{ine-integration}
\int_{B_1^+} \|\rho(f_{t}  ^z  h)v\|^{- \theta }d\mu(h)\le e^{-(\theta-\delta) \lambda t}.
	\end{equation}
\end{lemma}

%\begin{remark}\label{rem;new}
%	The same conclusion holds if we  only assume   $V$ has no nonzero vector fixed by $H$. In this case we can take  $\theta_0$ to be the minimum of those $\theta_0$ for each $V_i$ in
%	the decomposition (\ref{eq;complicate}).
%\end{remark}

\begin{proof}
	Without loss of generality, we assume further that the Euclidean norm $\|\cdot\|$ on $V$ satisfies the following  properties:
	\begin{itemize}
		\item 	Lyapunov subspaces of  $A_F$  are orthogonal to each other.
		\item  $H$-isotropic subspaces  $V_i\ (1\le i\le m)$ are orthogonal to each other.
	\end{itemize}

Let
	\begin{equation}\label{constant-c-r}
  R_i=\sup_{v\in V_i, \|v\|=1, h\in B^+_1}\|\pi_i(\rho(h)v)\|\quad \mbox{and } \quad R=\max\{ R_i: 1\le i\le m   \} .
	\end{equation}
	Let $C=\max\{C_i: 1\le i\le m \}$ where   $C_i$ is given  in   \eqref{constant-c-alpha}.  Let $\theta'=\max\{\theta_i: 1\le i\le m \}$.
	
	According to Lemma \ref{l-v+-},
	there exists  $T_{\frac{\delta}{{2\theta}}}>0$ such that
     (\ref{ine-v+-}) holds  	for any  $t\ge T_{\frac{\delta }{{2\theta}}}$,  any nonzero   $v\in V^{\lambda_i} \ (1\le i
     \le m)$  and  any
	$z\in B_1^c$, i.e.,
	\begin{align*}
	\|\rho(f_t  ^z  )v\|^{- \theta }\le e^{  -(1-{\frac{\delta}{{2\theta}}}) \theta
		\lambda_i t} \|v\|^{- \theta }
	\le  e^{-(\theta-\frac{\delta}{2} ) \lambda t}  \|v\|^{- \theta  }.
	\end{align*}
    This inequality and the assumption of the norm implies that for all nonzero $v\in V_i$ and $t\ge T_{\frac{\delta}{2\theta}}$
    	\begin{equation}\label{ine-f-t}
    \|\rho(f_{ t}  ^z  h)v\|^{ -\theta }\le  e^{-(\theta-\frac{\delta}{2})\lambda t }\|\pi_i(\rho(h)v)\|^{- \theta },
    \end{equation}
    where $\frac{1}{0}$ is interpreted as $\infty$.
	Let   $T_{\theta, \delta}\ge T_{\frac{\delta}{2\theta}}$ be a  large enough real number  so that
	 $t\ge T_{\theta, \delta }$ implies
	\begin{equation}\label{constant-t2}
	\frac{(2m)^{\theta'} CR^{\theta'-\theta}}{1-2^{\theta-\theta_0}} e^{-(\theta-\frac{\delta}{2})\lambda t }\le e^{-(\theta-\delta)\lambda t }.
	\end{equation}

	Let $v$ be a unit vector of $V$. We write $v=v_1+\cdots+v_m$ where $v_i\in V_i$.
	Since we assume different  $V_i$ are orthogonal to each other, there exists an integer  $i\in [1, m]$ such that $m\|v_i\|\ge \|v\|=1$.

	There is  a disjoint union decomposition of $B_1^+$ as   $$E(v_i,0)\cup\left ( \cup_{n\ge 0} E^{+}(v_i,2^{-n}R_i)\right),$$ where
	$$E^{+}(v_i,2^{-n}R_i)= E(v_i,2^{-n}R_i)\setminus E(v_i,2^{-n-1}R_i). $$
	Since $\mu (E(v_i, 0))=0$,  for any $ z\in B_1^c$ and
	$t\ge T_{\theta, \delta}$
	we have
	\begin{align*}
		%%%%%%%%%line one
	\int_{B_1^+}\|\rho(f_{ t}  ^z  h)v\|^{- \theta }d\mu(h) &\le \sum_{n=0}^{\infty}\int_{E^{+}(v_i,2^{-n}R_i)} \|\rho(f_{ t}  ^z  h)v_i\|^{-\theta }d\mu(h) \\
	%%%%%%%%%%line 2
	\text{(by \eqref{ine-f-t})}\quad &\le e^{-(\theta-\frac{\delta}{2}) \lambda t}\sum_{n=0}^{\infty}\int_{E^{+}(v_i,2^{-n}R_i)}\|\pi_i(\rho(h)v_i)\|^{- \theta }d\mu(h) \\
	%%%%%%%%%%%line 3
	\text{(by \eqref{constant-c-alpha})} \quad    &
	\le  e^{-(\theta-\frac{\delta}{2})\lambda t} \sum_{n=0}^{\infty} C_i 2^{\theta}(2^{-n}R_i)^{\theta_i-\theta} \|v_i\|^{-\theta_i} \\
	%%%%%%%%%%line 4
	&\le \frac{m ^{\theta'}2^{\theta'} CR^{\theta'-\theta}}{1-2^{\theta-\theta_0}} e^{-(\theta-\frac{\delta}{2}) \lambda t} \\
	%%%%%line 5
	\text{(by \eqref{constant-t2})} \quad       &\le e^{-(\theta-\delta) \lambda t}.
	\end{align*}
\end{proof}

\section{Eskin-Margulis height function}\label{sec;height}

 %$G\subset \SL_n(\RR)$ is a center-free semisimple Lie group
%without compact factors and $\Gamma=G\cap \SL_n(\ZZ)$ is a lattice of $G$.

Let the notation be as in Theorem \ref{thm;general}.
In this section, we will establish a
uniform contraction  property of the EM  height function   on $G/\Gamma $ with respect to a family of one parameter
groups $F_z  \ (z\in B_1^c)$.

Recall that $G/\Gamma=\prod_{i=1}^m G_i/\Gamma_i$ where each $G_i/\Gamma_i$ is a nonuniform  irreducible quotient of a semisimple Lie group without compact factors.
Since we assume the projection of $A_F$ to each $G_i$ is nontrivial, we have
$
H=\prod_{i=1}^m H_i$, {where } $H_i=G_i\cap H$ { is a connected   normal subgroup of }$G_i $
with positive dimension.

Let us recall the definition of the EM height function from \cite{em}. The EM height function is constructed on each $G_i/\Gamma_i$ using a finite set $\Delta_i$ of $\Gamma_i$-rational parabolic subgroups of $G_i$. Recall that  a parabolic subgroup $P$ of $G_i$
is $\Gamma_i$-rational  if the unipotent radical of $P$ intersects $\Gamma_i$ in a lattice.
If the rank of $G_i$ is bigger than one, then Margulis' arithmeticity theorem implies that there is a $\QQ$-structure on  $G_i$ such that $\Gamma_i$ is commensurable with $G_i(\ZZ)$. In this case the set
 $\Delta_i$ consists of  standard $\QQ$-rational maximal parabolic subgroups of $G_i$ with respect to a fixed $\QQ$-split torus and fixed positive roots. So the irreducibility of $\Gamma_i$
 implies that no conjugates of $H_i$ is contained in any $P\in \Delta_i$.
 The same conclusion holds in the case where
 $G_i$ has rank one. The reason is that  in this case  $H_i=G_i$ and $\Delta_i=\{P \}$ where $P$ is a maximal parabolic subgroup defined over $\RR$.

For each $P_{i, j}\in \Delta_i$, there exists a representation  $\rho_{i,j}: G_i\to \GL(V_{i,j}) $ on a normed vector space  and a nonzero  vector $w_{i,j}\in V_{i,j}$ such that the stabilizer of $\RR w_{i,j} $ is $P_{i,j}$.
We consider  $\rho_{i,j}$ as a representation of  $G$ so that $\rho(G_s)$ is the
identity linear map  if $s\neq i$.
Let $V_{i,j}^{H}$ be the
$H$-invariant
 subspace of $V_{i,j}$ consisting of $H$-invariant vectors. Let $\pi_{i,j}$ be the projection of $V_{i,j}$ to the $H$-invariant subspace
$V_{i,j}' $   complementary to  $V_{i,j}^{H}$.
Since no conjugates of $H_i$ is contained in $P_{i,j}$ and $G_i=K_{i} P_{i,j}$ for some maximal compact subgroup $K_i$ of $G_i$, there exists $C\ge 1 $  such that
\[
 \|v\|\le  C \|\pi_{i,j} (v)\|
\]
for all $v\in \rho_{i,j}(G) w_{i,j}$.
Note that  $V_{i,j}'$ is
 $G$-invariant and it  has no nonzero $H$-invariant vectors.
Therefore, Lemma \ref{l-contraction} implies the following lemma which  corresponds to  {\bf Condition A} in \cite{em}.

\begin{lemma}
	\label{lem;crutial}
	For each pair of  index $i,j$
	there exist positive constants  $\theta_0^{i,j}$ and $\lambda^{i,j}$ such that for any  $0<\delta<\theta<\theta_0^{i,j}$, any nonzero $v\in\rho_{i,j}( G) w_{i,j}$ and any $z\in B^c_1$ one has
	\begin{align}\label{eq;kao}
\int_{B _1^+}\|\rho_{i, j} (f_t  ^z  h) v  \|^{-\theta}\dd h \le
 e^{-(\theta-\delta)t\lambda^{i,j} } \|v\|^{-\theta}
	\end{align}
	provided $t\ge T_{\theta, \delta }^{i,j}$
	 where $T_{\theta, \delta}^{i,j}>0$ is a constant  depending on $\theta$ and $\delta $.
\end{lemma}
\begin{proof}
We assume without loss of generality that for all $V_{i,j} $ the norm $\|\cdot\|$ is  Euclidean and   $V_{i,j}^H$ and $V_{i,j}'$ are orthogonal to each other.
According to Lemma \ref{l-contraction}, for each representation $\rho_{i,j}|_{V_{i,j}'}$, there exist positive constants
$\theta^{i,j}_0$ and $ \lambda^{i,j}$ with the following properties:  for any $0<\delta<\theta<\theta^{i,j}_0$ there exists $T_{ \theta, \delta}>0$ such that for
any $t\ge T_{\theta, \delta}, z\in B_1^c$ and any nonzero  $v\in \rho_{i,j}( G) w_{i,j}$,
one has
\begin{align*}
%%%%%%%%%line 1
\int_{B _1^+}\|\rho_{i, j} (f_t  ^z  h) v  \|^{-\theta}\dd h& \le
\int_{B _1^+}\|\rho_{i, j} (f_t  ^z  h) \pi_{i,j}(v ) \|^{-\theta}\dd h  \\
%%%%%%%%line 2
&\le e^{-(\theta-\delta)t\lambda^{i,j} } \|\pi_{i,j}(v)\|^{-\theta} \\
%%%%%%%%line 3
&\le C^{\theta} e^{-(\theta-\delta)t\lambda^{i,j} } \|v\|^{-\theta}.
\end{align*}
It is not hard to see from above estimate  that (\ref{eq;kao}) holds for sufficiently large $t$.
%Therefore the lemma holds for
%$$\theta_0=\min_{i,j}\theta^{i,j}_0 \quad \mbox{and }\quad \lambda=\min_{i,j}\lambda^{i,j}   . $$
\end{proof}

Besides $\rho_{i,j}$,
the EM height function is  constructed  using  positive  constants $c_{i,j}$ and $ q_{i,j}$ which are combinatorial data determined  by the root system, see \cite[(3.22),(3.28)]{em}.
 Let
\begin{align}\label{eq;simple}
u _{i,j}(g\Gamma)= \max _{\gamma\in \Gamma } \frac{1}{\|\rho_{i,j} (g\gamma){w_{i,j}} \|^{1/c_{i,j} q_{i,j}}}
%\quad
%\mbox{ and }\quad  u_{i,j}(g\Gamma)&= u _{i, j}(g\Gamma)^{1/q_{i,j}}\quad
\end{align}
{where } $g\in G$.\footnote{Although only the product $c_{i,j}q_{i,j}$ is used in this paper,  the constants $c_{i,j}$ and $q_{i,j}$ are given by different combinatorial data and we use both of them for the consistency with \cite{em}.}
Let
\begin{equation}\label{eq;alpha1}
\begin{split}
\theta_1=\max\{\theta>0:  \frac{\theta}{q_{ij}c_{i,j}}\le \theta_0^{i,j} \quad \mbox{for all }i,j\}
\quad \mbox{ and }\quad
\alpha_1=\min_{ i,j } \{ \frac{\theta_1}{q_{ij}c_{i,j} }\lambda^{i,j}\},
\end{split}
\end{equation}
where $\theta_0^{i,j}$ and $\lambda^{i,j}$ are constants given by Lemma \ref{lem;crutial}.
We call $\alpha_1$ a contraction rate for the dynamical system $(G/\Gamma, F^+)$.
\begin{remark}
  We will see in next sections that $\alpha_1$ plays an important role in bounding the Hausdorff dimension of $\D(F^+,\ggm)$. We believe that optimal $\alpha_1$ is  possible to give the sharp bound of the dimension. By Remark \ref{rem-3}, the constant $\theta_{i,j}$ can be explicitly calculated, so are the constants $\theta_1$ and $\alpha_1$. Consequently, it will be clear in the proof in the next sections that the upper bound of the dimension we obtain can also be explicitly calculated, although not optimal.
\end{remark}

\begin{lemma}
	\label{lem;key}
	For every $\alpha< \alpha_1$, there exist $0<\theta<\theta_1$ and $T>0$ such that for
	all $t\ge T$ and $\epsilon $ sufficiently small depending on $t$, the EM height function
	\begin{align}
	\label{eq;u}
	u: G/\Gamma \to (0, \infty)\quad \mbox{defined by }	u(x)=\sum_{i,j}(\epsilon\,  u_{i,j}(x))^{\theta}
	\end{align}
 satisfies the following properties:
		\begin{enumerate}
		\item $u (x)\to \infty$ if and only if $x\to \infty $ in $G/\Gamma$.
%		\item $u$ is positive and bounded below away from zero.
		\item For any compact subset $K$ of $G$, there exists $C\ge 1$ such that
		$u(hx)\le C u(x)$ for all $h\in K$ and $x\in G/\Gamma$.
		\item There exists $b>0$ depending on $t$ such that for all  $z\in B_1^c$ and
		$x\in G/\Gamma$ one has
		\begin{align}\label{eq;tech}
		\int_{B_1^+} u(f_t  ^z   hx) \dd \mu(h)< e^{-\alpha t} u(x)+b.
		\end{align}
		\item  There exists $\ell \ge 1$ such that if $u(x)\ge \ell $, then for all $z\in B_1^c$
		\begin{align}\label{eq;contract}
				\int_{B_1^+} u(f_t  ^z   hx) \dd \mu(h)< e^{-\alpha t} u(x).
		\end{align}
	\end{enumerate}
\end{lemma}
\begin{proof}
	It  follows from the corresponding results for each $G_i/\Gamma_i$
	proved in \cite{em} that
	the first two conclusions hold for any choice of $\theta$ and $\epsilon$.  Note that (4) is a direct corollary of (3).

Now we prove (3). 	Let $n$ the cardinality of the indices  $i,j$ appeared in the definition of $u$.
	We fix $\delta >0$ sufficiently small such that $$
%	e^{-\delta}(n+1)\le 1\quad\mbox{ and }	\quad
	\alpha+\delta +\frac{\delta \lambda^{i,j}}{c_{i,j} q_{i,j}}<\alpha_1\quad \forall \ i,j.$$
According to the definitions of $\theta_1$, $\alpha_1$ and the choice of $\delta $ above,		 there exists $\theta>0$  such that
\[
\theta<\theta_1\quad \mbox{and } \quad\frac{(\theta-\delta) \lambda^{i,j}}{c_{i,j} q_{i,j}}\ge \alpha +\delta\quad \forall\  i,j.
\]
Let  $\delta_{i,j}=\delta/ c_{i,j} q_{i,j}$, $\theta_{i,j}
=\theta/ c_{i,j} q_{i,j}$, then according to Lemma \ref{lem;crutial} there exists $T^{i,j}>0$
such that for $t\ge T^{i,j}$ one has (\ref{eq;kao}) holds with $\delta=\delta_{i,j}$
and $\theta=\theta_{i,j}$.
We will show that Lemma \ref{lem;key} holds for  $T=\frac{\log 2}{\delta }+\max_{i,j} T^{i,j}$.

Now we fix $0<\epsilon<1$, $x=g\Gamma\in G/\Gamma, t\ge T$ and $i, j$.  According to the definition of $u_{i,j}(x)$, there exists $\gamma\in \Gamma$ such that
$u_{i, j}(x)=\frac{1}{\|\rho(g \gamma) w_{i,j}  \|^{1/c_{i,j}q_{i,j}}}$. For any $h\in B_1^+$
and $z\in B_1^c$,
if $u_{i,j}(f_t^zhx)=\frac{1}{\|\rho(f_thg \gamma) w_{i,j}  \|^{1/c_{i,j}q_{i,j}}}$, then we can use (\ref{eq;kao}). Otherwise, there exist $0<\kappa<1$, $b>0$ and $C'\ge 1$ where $b$ and $C'$
depend on $t$ such that
\[
(\epsilon u_{i,j}(f_t^zhx))^\theta\le C' \epsilon^\kappa (\epsilon u(x))^\theta+\frac{b}{n}.
\]
These facts are proved in \cite[\S 3.2]{em}.
In summary,   we have
	\begin{align*}
	%%%%%%%line 1
	\int_ {B_1^+}    (    \epsilon u_{i, j}(f_t h x)      )^\theta\dd h    & \le
  \epsilon^\theta 	\int_{B_1^+ }  \frac{1}{\|\rho_{i,j}(f_t h g \gamma) w_{i,j}  \|^{\theta /c_{i,j}q_{i,j}}} \dd h  + \epsilon ^\kappa C' u(x)+\frac{b}{n} \\
  %%%%%%%lien 2
	&\le e^{-(\theta-\delta)    t \lambda^{i,j} /c_{i,j} q_{i,j} }     (  \epsilon u_{i,j}(x))^\theta  + \epsilon ^\kappa C' u(x)+\frac{b}{n}\\
	&\le e^{-(\alpha +\delta)t}(  \epsilon u_{i,j}(x))^\theta + \epsilon ^\kappa C' u(x)+\frac{b}{n}.
	\end{align*}

Therefore, we have
	\[
		\int_ {B_1^+}        u(f_t h x)      \dd h     \le
 e^{-(\alpha +\delta)t} u(x) + n\epsilon ^\kappa C' u(x)+b.
	\]
	We choose $\epsilon $ sufficiently small so that $n\epsilon ^\kappa C' \le e^{-(\alpha+\delta) t}$, then (\ref{eq;tech}) holds.
	% Note that (4) is a corollary of (3).
\end{proof}

\section{Applications of the uniform contraction property}\label{sec;first}

 In this section we will introduce and study some auxiliary sets closely related to $\D(F^+, \ggm)$ using the uniform contraction property of the EM height function established in Lemma \ref{lem;key}. To be specific, we will prove some covering results for these auxiliary sets in Proposition \ref{l-main2} and these covering results will play an important role in bounding the Hausdorff dimension of $\D(F^+, \ggm)$.

  Let
 $\alpha_1$ be a contraction rate of the dynamical system $(G/\Gamma, F^+)$
 given by (\ref{eq;alpha1}) and let  $\lambda$ be the top Lyapunov exponent of $A_F$ in the representation $(\Ad, \mathfrak g)$.
We fix $\alpha<\alpha_1,t >0$ and a EM height function
 $u: G/\Gamma \to (0, \infty)$ so that  Lemma \ref{lem;key} holds.
 Let $\ell \ge 1$ so that  (\ref{eq;contract}) holds for all $z\in B_1^c$ if $u(x)\ge \ell$.
  By Lemma \ref{lem;key} (3), there exists $C\ge 1$ such that
 \begin{align}\label{eq;c}
 C^{-1} u(x)\le u(f_s hx)\le C u(x) \quad \mbox{for all } 0\le s\le t, h\in B_2^G
 \mbox{ and } x\in G/\Gamma.
 \end{align}

We also fix an auxiliary
 $\delta>0$ (which will go to zero finally)
and assume that   $t$ is sufficiently large so  that
 according to Lemma \ref{l-v+-}
 { for all } $r\le 1, z\in B_1^c$
 \begin{align}
 B_{e^{-(\lambda+\delta) t} r}^+ \subset  f_{-t}  ^z  B_r^+f_{t}  ^z   &\subset B_{r/4}^+; \label{l-t4}\\
 B_{e^{-\delta t} r}^c\subset f_{-t}  ^z  B_r^cf_{t}  ^z   &\subset  B_{e^{\delta t} r}^c;\label{eq;B0}\\
 \label{eq;details}
 2&< e^{\delta (\alpha+1) t/2}.
 \end{align}

Note that the logarithm map from  the metric space   $(B_1^+, \dist)$ to the Lie algebra  $\mathfrak g^+$ (with the fixed Euclidean structure)
is a   bi-Lipschitz homeomorphism    to its image. Therefore $(B_1^+, \dist)$ is Besicovitch, see \cite{mat}, namely, for any subset $D$ of $B_1^+$ and a covering of $D$ by balls centered at $D$, there is a finite sub-covering such that each element of $D$ is covered by at most
$E'$ times.
Therefore, there exists $E\ge E' $ such that for all $0<r\le 1$, the set
$B^+_{1/2}$ can be covered by no more than  $E r^{-d}$ open balls of radius $r$, where $d=\dim G^+$.

We use $|I|$ to denote the cardinality of a finite set $I$. The following is the
main result of this section.

\begin{prop}\label{l-main2}	
Let $x\in G/\Gamma$.  There exists
$0<\sigma<1$ and $E_0\ge 1$   such that for
%  any $x\in K$,
 $ z\in B_1^c$ and  $N\in \NN$,
 the set
\begin{equation}\label{def-d+}
\D_x(z, N, \sigma,  C^2\ell )\df \{h\in B_{1/2}^+ : |\{1\le n\le N : u (f^z_{ nt}
hx)\ge C^2\ell  \}|\ge \sigma N\}
\end{equation}
%$$\D_x(F_{ w }^+, N, \sigma; L_1)\cap B^{+}_{1/2}$$
can be covered by no more than $E_0e^{(d\lambda-{\alpha}+\delta( d+\alpha)) tN}$ open balls of radius $e^{-(\lambda+\delta) t N}$  in $B_1^+$.
 \end{prop}
The rest of this section is devoted to show that Proposition \ref{l-main2} holds for
\begin{align}
\label{eq;sigma}
\sigma& =\frac{(1-\delta/2)\alpha t +\log C}{\alpha t +\log C}
%\max\{1-\frac{\delta}{2}, 1-\frac{\delta \alpha t}{2} \}.
\end{align}
In the rest of this section we fix $ z\in B_1^c$ and $N\in \NN$. We begin with the following simple observation.
\begin{lemma}
	\label{lem;add}
	%	Let $L_1=CL$ and $\D_x(F_{ w }^+, N, \sigma; L_1)=\{h\in B_{1/2}^+ : |\{1\le n\le N : u (f_{ nt}
	%	hx)\ge L_1  \}|\ge \sigma N\}	$.
	If  $B\subset  G^+$ is a ball of radius $e^{-(\lambda+\delta) t N}$ centered at  $\D_x({ z }, N, \sigma , C^2\ell )$, then $B
%	\cap B_{1/2}^+
	\subset \D_x({ z }, N, \sigma,  C\ell )$.
\end{lemma}

\begin{proof}
	Let $h_0$ be the center of $B$ and $h\in B$.  It suffices to  show that for all   $1\le n\le  N$
  if  $u (f_{nt}^z h_0x)\ge C^2\ell $ then $u (f_{nt}^z h x)\ge C\ell $.
 By (\ref{l-t4})  we have
	\[
	\dist (f_{nt}^z h_0, f_{nt}^zh )=\dist (1_G, f_{nt}^zhh_0^{-1} f_{-nt}^z)< 1.
	\]
	By (\ref{eq;c})
	\[
	u (f_{nt}h x)=u (f_{nt}hh_0^{-1} f_{-nt}\cdot f_{nt}h_0 x)\ge C ^{-1}u (f_{nt}h_0 x)\ge C^{-1}\cdot C^2\ell = C\ell.
	\]
	
\end{proof}

For a subset  $I\subset \{ 1, \ldots, N \}$, we let
\begin{equation}\label{def-d+i}
\D_x(  z  ,  I,  C\ell  )=\{h \in B_{1/2}^+ :  u (f_{ nt}  ^z
hx)\ge C \ell   \mbox { for all } n\in I \}.
\end{equation}
Since
$
\D_x({ z }, N, \sigma,  C\ell )=\bigcup_{|I| \ge \sigma N} \D_x({ z },  I,  C\ell )
$,
one has
\begin{align}\label{eq;sum}
\mu(\D_x({ z }, N, \sigma,  C\ell ))\le \sum_{|I| \ge \sigma N}
\mu (\D_x({ z },  I, C\ell )).
\end{align}

The following lemma will play an important role in the proof of Proposition \ref{l-main2}.
\begin{lemma}\label{l-int-ine}
	Suppose that  $I\subset \{1, \ldots, N \}$  and  $|I| \ge \sigma N$. Then
	\begin{align}\label{eq;con}
	\mu (\D_x(z,  I, C\ell ) )\le {C^2}
%	{\ell ^{-1} }
	 u (x) e^{-(1-\delta/2)\alpha tN  }.
	\end{align}
	
\end{lemma}

 We fix $I$ as in the statement of Lemma \ref{l-int-ine}.
Our strategy is  to estimate the measure of $\D_x({ z },  I,  C\ell )$
by
 relating it to a subset coming from random walks on $G/\Gamma$ with alphabet   $f_t  ^z   B_1^+$.
Let   $p =\sup I$ and  for $1\le k\le p$ let
\begin{equation*}
	Z_k=\{(h_1,\ldots,h_{k })\in (B_1^+)^{k }: u  (f  ^z  _{ t} h_n\ldots f  ^z  _{t} h_1x)
	\ge \ell \ \forall\   n\in (I\cap[1,k])\}.
\end{equation*}
Define $\eta :(B_1^+)^p \rightarrow G^+$ by
\begin{align}\label{eq;h}
\eta(h_1,\ldots, h_{p })= \tilde{h}_{p } \cdots \tilde{h}_{1}, \text{ where } \tilde{h}_{n}= f  ^z  _{-(n-1)t}h_n f  ^z  _{(n-1)t}.
\end{align}
We remark here that the image  of $\eta$ is contained in $B_2^+$ by (\ref{l-t4}). The following two lemmas are needed in the proof of Lemma \ref{l-int-ine}.

\begin{lemma}\label{lem;contain}
For all $h\in \D_x({ z }, I,  C\ell )$ one has $\eta^{-1}(h)\subset Z_p$.
\end{lemma}

\begin{proof}
Suppose that  $\eta(h_1, \ldots, h_p)=h$ where $h_i\in B_1^+$. Then for all $n\le  p$
\[
\dist(f_{nt}^zh, f_{t}^zh_n\cdots f_t^z h_1)=\dist(f_{nt}^z\tilde h_{p }\cdots \tilde h_{n+1}f_{-nt}^z, 1_G)< 2,
\]
where we use (\ref{l-t4}), (\ref{eq;h}) and the right invariance of $\dist(\cdot, \cdot)$.
Therefore  by (\ref{eq;c}) we have  for $n\in I$
\begin{align*}
	u (f_{t}^zh_n\cdots f_t^z h_1x)\ge C^{-1}u (f_{nt}^zhx)\ge \ell .
\end{align*}
So $(h_1, \ldots, h_p)\in Z_p$ and the proof is complete.
\end{proof}

Let  $\widetilde \mu_{ n}$ be the Radon  measure on $G^+$ defined by
\begin{align}\label{eq;def}
\int_{G^+}\varphi(h) \dd\widetilde \mu_{ n}(h)=\int_{B_1^+}
\varphi(f_{-nt}  ^z  hf_{nt}  ^z  ) )\dd h
\end{align}
for all $\varphi\in C_c(G^+)$.
%Let $\mu_{ 1}=\widetilde \mu_{ 0}$ and inductively define $\mu_{n}$ by
For any positive integer $n$ let
$\mu_{n}=\widetilde \mu_{ n-1}*\cdots *\widetilde \mu_{1}*\widetilde \mu_{0}$ be the measure on $G^+$ defined by the
$n$ convolutions.
Clearly, $\mu_n$ is absolutely continuous with respect to $\mu$ and $\mu_p$
is the pushforward of $(\mu|_{B_1^+})^{\otimes p}$ by the map $\eta$.
The following lemma shows that $\mu_n$ has density bigger than or equal to  one  at every $h\in B_{1/2}^+$.

\begin{lemma}\label{l-rd-conv}
For all  $n\le N $ and $ h\in B_{{1}/{2}}^+$ we have
$
\frac{d\mu_{n}}{d\mu}(h)\ge 1 .
$
\end{lemma}

\begin{proof}
The conclusion is clear if $n=1$. Now we assume $n>1$ and let
$$\nu= \widetilde \mu_{ n-1}* \widetilde \mu_{n-2}* \cdots  * \widetilde \mu_{ 1}.$$
It follows from (\ref{l-t4}) and (\ref{eq;def}) that for $k>0$ the probability measure $\widetilde \mu_{ k}$ is supported on $B_{1/4^k}^+$. Since the metric on $G^+$ is right invariant, the measure $\nu$ is supported on $B_{1/2}^+$.
Suppose $\nu=\varphi \dd \mu$, then $\mu_{  n}=\nu * \widetilde \mu_0 =\varphi* \mathbbm{1}_{B_1^+} \dd \mu$. So for any $h\in B_{1/2}^+$, we have
$$\varphi* \mathbbm{1}_{B_1^+}(h)=\int_{G^+}\varphi(h_1)\mathbbm{1}_{B_1^+}(h_1^{-1}h) d\mu(h_1)\ge \int_{B_{1/2}^+}\varphi(h_1)d\mu(h_1) =1.$$
\end{proof}

Now we are ready to prove Lemma \ref{l-int-ine}.
\begin{proof}[Proof of Lemma \ref{l-int-ine}]
	
 By Lemmas \ref{lem;contain} and  \ref{l-rd-conv},
  	\begin{align}\label{eq;temp}
  	\mu (\D_x(z,  I,  C\ell ) )\le \mu_\ell(\D_x(z,  I,  C\ell ) )\le
  	\mu_p(Z_p).
  	\end{align}
  	Now we are left to estimate $\mu_p(Z_p)$. For $1\le k \le p$ let
  \begin{equation}\notag
  s(k )=\int _{Z_k } u (f_{ t}  ^z   h_{ k }\cdots f_{ t}  ^z   h_1x)d\mu^{\otimes k }(h_1, \cdots, h_{ k}).
  \end{equation}
  Let
  \begin{align}\label{ine-iterated int}
  s(p+1)=\int_{Z_{p}}\left[
  \int_{B_1^+} u (f_{ t}  ^z   h_{p+1  }f_t^z h_{p }\cdots f_{ t}  ^z   h_1x) \dd\mu(h_{p+1} )\right ]d\mu^{\otimes p }(h_1, \cdots, h_{p}).
  \end{align}
  Then for every $1<  k\le p+1$,
\begin{align}\notag
%\label{eq;inf}
	s(k )&\le \int_{Z_{k-1}}\left[
	 \int_{B_1^+} u (f_{ t}  ^z   h_{k  }f_t^z h_{k-1}\cdots f_{ t}  ^z   h_1x) \dd\mu(h_k )\right ]d\mu^{\otimes (k -1)}(h_1, \cdots, h_{k -1}).
\end{align}

If $k -1\in I$, then  $s(k )\le e^{-\alpha t}s(k -1)$ by (\ref{eq;contract}). If $k -1\not \in I$, then by (\ref{eq;c}) we have  $s(k )\le C s(k -1)$.
We apply this estimate to   $k=p+1, p, \cdots, 2$, then we have
 \begin{align*}
 	  	s(p+1)\le C^{(N-|I|)}e^{-|I| \alpha t} \int_{B_1^+} u (f_t hx\dd \mu (h))\le    C^{1+(1-\sigma) N }e^{-\sigma\alpha  t N }u (x)
 	 % 	\le   C^{\sigma N} e^{-(1-2\sigma)^2 \alpha  tN}u (x)
 	 .
 \end{align*}
The choice of $\sigma$ in (\ref{eq;sigma}) implies that
 \begin{align}\label{eq;temp2}
 		s(p+1)\le C e^{-(1-\delta/2) \alpha t N}u (x).
 	\end{align}

 On the other hand, in view of  (\ref{ine-iterated int}),  (\ref{eq;c}) and the fact $p=\sup I $
 we have   \begin{align}\label{eq;che}
 s(p +1)\ge C^{-1} s(p)\ge C^{-1}\ell \cdot
 \mu_p(Z_p).
 \end{align}
  Therefore,  (\ref{eq;con}) follows from (\ref{eq;temp}), (\ref{eq;temp2}) and (\ref{eq;che}) and the observation $\ell \ge 1$.

 \end{proof}

\begin{comment}
\begin{remark}\label{rem-h0}
It is not hard to see that there also exists $K^c\subset G^c$ such that an analogue of Lemma \ref{def-c} holds. Indeed, the only thing we should take care of when repeating the proof is that the map $\exp^{-1}$ may be only well-defined on an open neighborhood of $1$. Indeed, it is easily seen that to prove such an analogue, we only need to consider $n$ large enough. As such, we may choose $n$ large enough so that $\exp^{-1}$ is well-defined on $B^c(1, e^{-2\lambda   nT})$ to pass to the study of action on Lie algebras.
\end{remark}
\end{comment}

\begin{proof}[Proof of Proposition \ref{l-main2}]
	As  before we fix  $ z$ and $ N$ as in the statement. Let $\sigma$ be as in  (\ref{eq;sigma}).
	Since $(B_1^+, \dist)$ is Besicovitch, there exists a covering  $\fU$
	of $\D_x({ z }, N, \sigma,  C^2\ell )$ by open balls of radius $e^{-{(\lambda+\delta) t N}}$ centered at $\D_x({ z }, N, \sigma, C^2\ell )$
	such that each element of $\D_x({ z }, N, \sigma, C^2 \ell )$ is covered by at most
	$E$ times.  By Lemma \ref{lem;add}, each $B\in \fU$ is contained in
	$\D_x({ z }, N, \sigma,  C\ell )$, so  in view of (\ref{eq;c+})
	\begin{align}\label{eq;tmd}
	\mu(	\D_x({ z }, N, \sigma, C\ell ))\ge\frac{ |\fU |}{E} \mu(B^+_{e^{-(\lambda+\delta) t N}})\ge \frac{ |\fU |}{C_0E} e^{-(\lambda+\delta) d  t N} .
	\end{align}
	
	On the other hand,
  since there are $2^N$ subsets  $I\subset \{1, \ldots, N \}$, by (\ref{eq;sum}), (\ref{eq;details}) and  Lemma \ref{l-int-ine},
  % $2\le e^{\delta \alpha  t/2}$ in ,
  we have
  \begin{equation}\label{mu-dx+1}
  \mu\big(\D_x({ z }, N, \sigma,  C\ell )\big) \le C^2
  %\ell ^{-1}
   2^Ne^{-(1-\delta/2) \alpha tN}u (x)\le e^{-(1-\delta)\alpha tN}u (x).
  \end{equation}
  By (\ref{eq;tmd}) and (\ref{mu-dx+1}),
  \[
  |\mathfrak U| \le u(x) C_0C^2 E\cdot    e^{ ( d\lambda -\alpha+\delta  (d+\alpha)) t N}.
  \]
The conclusion now follows by taking
  $
  E_0=u(x)C_0 C^2 E.
  $

\end{proof}

\section{Upper bound of Hausdorff dimension}

In this section, we finish the proof of Theorem \ref{thm;general}. We will use the same notation as in
\S \ref{sec;first}   prior to Proposition \ref{l-main2}.
For $(z,h)\in B_1^c B_1^+ , \ell '>0 $ and $N\in \NN$,  let
$I_{N, \ell'}(z,h)$ denote the set of $n\in \{1,\ldots,N\}$ satisfying
$  u(f_{nt}zhx)\ge \ell '
$.
For  $x\in \ggm$, let
\begin{equation}\label{def-d+0}
\D^{0}_x(F^+, N, \sigma, \ell')=\{(z,h)\in B_{1/2}^c B_{1/2}^+ : |I_{N, \ell '}(z,h)|\ge \sigma N\}.
\end{equation}

\begin{lemma}\label{l-main}
Let $x\in \ggm$. Then there exist $0<\sigma<1$ and   $E_2\ge 1$     such that for any
$N\in \NN$
 the set $\D^{0}_x(F^+, N, \sigma, C^4 \ell )$ can be covered by  no more than $E_2e^{(d^c \lambda+d \lambda -\alpha+\delta (d^c+d+\alpha)) tN}$ open balls of radius
  $e^{-(\lambda+\delta) t N}$ in $G^cG^+$.
\end{lemma}

\begin{proof}

Let  $0<\sigma<1$ and   $E_0\ge 1 $  so that Proposition
\ref{l-main2} holds.
We fix $N\in \NN$.

We  claim that:
for $W=B_{e^{-(\lambda+\delta) tN}}^c\cdot z\subset B_{1}^c$, we have
\begin{equation}\label{claim-w}
\Big( \D_x^{0}(F^+,N,\sigma, C^4\ell )\cap (W B_{1}^+)\Big)\subset \Big(W  \D_x(z  , N, \sigma, C^2\ell )\Big).
\end{equation}
Let $(z_1,h_1)\in W B_{1}^+$. Suppose that $1\le n\le  N$ and
$
u(f_{nt}z_1h_1x)\ge C^4\ell
$.
In view of (\ref{eq;B0}) and (\ref{eq;c}) we have
\begin{equation*}
\begin{split}
u(f_{ nt}^zh_1x)=u(z^{-1}f_{nt}zh_1x)\ge C^{-1}  u(f_{nt}zh_1x)  \\
=C^{-1}u(f_{nt}(zz_1^{-1})f_{-nt}\cdot  f_{nt}z_1h_1x)\ge C^{-2}u(f_{nt}z_1h_1x)\ge  C^2\ell .
\end{split}
\end{equation*}
%Using (\ref{eq;c}) and the above inequality, we get
%\begin{equation*}
%u(f_{ nt}^zh_1x)=u(z^{-1}f_{nT}zh_1x)\ge C^{-1}u(f_{nt}zh_1x) \ge L_0.
%\end{equation*}
In other words, we have proved that if  $n \in I_{N, C^4\ell }(z_1,h_1)$, then $u(f_{nt}^z h_1 x)\ge C^2\ell $.  Therefore, if $(z_1, h_1)$ belongs to the left hand side of   \eqref{claim-w} then it also belongs to the right hand side.

Since  $(B_1^c, \mathrm{dist})$ is also Besicovitch, there exists
$E_1\ge 1$ such that for all $0<r\le 1$, $B_{1/2}^c$ can be covered by no more than
  $E_1 r^{-d^c }$ open balls of radius $r$.
We fix a  cover  $\fU^c$ of $B_{1/2}^c$ that consists of open  balls of radius $e^{-(\lambda+\delta ) Nt}$  with  $|\fU^c|\le E_1 e^{d^c(\lambda+\delta) Nt}$.
%  \begin{itemize}
 %   \item $|\fU^c_N|\le E_1 e^{d^c(\lambda+\delta) nt}$,
 %   \item Any $B\in \fU'_N$ is contained in $ B_1^c$.
    %\item each point in $B_1^c$ is covered by at most $D_2'$ balls from $\fU'_n$.
 % \end{itemize}
 We assume each element of $\fU^c$ has nonempty intersection with $B_{1/2}^c$, then it is contained in $B_1^c$ in view of (\ref{eq;details}).
Let  $W_z\in \fU^c$ be a ball  centered at $z\in B_1^c$. Proposition \ref{l-main2} implies that  there exists a covering   $\fU_{ z}$ of
$\D_x(z  , N, \sigma, C^2\ell ) $ by open balls of radius $e^{-(\lambda+\delta)t N}$ such that
\begin{align*}
%\label{eq;fish}
|\fU_{ z}|\le E_0 e^{d\lambda-\alpha+\delta(d+\alpha)}.
\end{align*}
In view of  claim \eqref{claim-w},
 the following class of sets
\begin{align}
\label{eq;W}
\{W_z  B: W_z\in \fU^c, B\in \fU_{ z}\}
\end{align}
 forms an open cover of $\D^{0}_x(F^+, N, \sigma, C^4 \ell)$.
 It is easily checked that there exists $E_1'\ge 1$ not depending on $N$ such that each element $W_zB$ of (\ref{eq;W})   can be covered by $E_1'$ open balls of
 radius $e^{-(\lambda+\delta)Nt }$ in $G^cG^+$.
 %It follows from Lemma \ref{l-main2} and \eqref{claim-w} that, the set $\D_x^{0}(F^+,N,\sigma; L_2)$ can be covered by no more than
 %$$E_1e^{d^c\lambda  tN}\times E_0 e^{(d\lambda-\alpha + \delta \alpha)  tN}=E_0E_1e^{(d^c\lambda+d \lambda-\alpha +\delta c )   tN}. $$
 %balls of the form $B^c(z,e^{-\lambda tN})\times B^+(h, e^{- \lambda tN})$.
Therefore the lemma holds with $E_2=E_0E_1E_1'$.
\end{proof}

\begin{theorem}
	\label{thm;simple}
	For any $x\in G/\Gamma$, the Hausdorff dimension of $ \D_x^{0}\df \{(z,h)\in B_{1/2}^c B_{1/2}^+:
	 zh x \in  \D(F^+, G/\Gamma) \}$ is at most $d^c+d -\frac{\alpha_1}{\lambda }$.
\end{theorem}

\begin{proof}
For each $\alpha<\alpha_1$ and $0<\delta<1$ we first choose $t>0$, a height function $u$ and $\ell, C \ge 1$ so that
Lemma \ref{lem;key}, (\ref{eq;c}), (\ref{l-t4}), (\ref{eq;B0}) and (\ref{eq;details}) hold. Then
%for   $\delta$ sufficiently small
there exists $0<\sigma<1$ and $E_2\ge 1 $ so that Lemma
\ref{l-main} holds.

It follows  from  Lemma \ref{lem;key} (1)(2) and the definition  of   $\D(F^+, G/\Gamma)$ that
$$\D_x^{0}\subset \bigcup_{M\ge 1} W_M\quad \mbox{where}\quad
W_{M}=
%\df
\bigcap _{N\ge M} \D^{0}_x(F^+, N, \sigma, C^4 \ell ).
  $$ 	

Recall that for any metric space $S$,
$$\dim_H S=\inf\left\{s>0: \inf_{\{B_i\}} \sum \rho(B_i)^s=0\right\},$$
where the latter $``\inf"$ is taken over all the countable  coverings $\{B_i\}$ of $S$ that consist of open  metric balls.
Then  in view of Lemma \ref{l-main}, we have
\begin{align*}
\dim_HW_M&\le \liminf_{N\to \infty}  \frac{[d^c\lambda+d\lambda -\alpha+\delta(d+d^c+\alpha)]t N+ \log E_2 }{\lambda t N}\\
& = d^c+d -\frac{\alpha}{\lambda}+\delta \frac{d+d^c+\alpha}{\lambda}.
\end{align*}
Therefore
\[
\dim_H \D_x^{0}\le d^c+d -\frac{\alpha}{\lambda}+\delta \frac{d+d^c+\alpha}{\lambda}.
\]
The conclusion follows by   first letting  $\delta\to 0$ and then letting  $\alpha\to \alpha_1$.
	
\end{proof}

\begin{lemma}
	\label{lem;stable}
	If $x\in \mathfrak D(F^+, G/\Gamma)$ and $h\in G^-$, then
	$hx \in \mathfrak D(F^+, G/\Gamma) $.
\end{lemma}

\begin{proof}
	Note that  by Lemma \ref{l-v+-},
	\[
	\dist (f_t hx, f_t x)\le \dist(f_t h f_{-t} , 1_G)\to 0
	\]
	as $t\to \infty$. Therefore the lemma holds.
\end{proof}

\begin{proof}[Proof of Theorem \ref{thm;general}]
We will show that \[
\dim_H \D(F^+, G/\Gamma)\le d^-+d^c+d -\frac{\alpha_1}{\lambda}.
\]	
In view of the  local nature of Hausdorff dimension and the definition of the metric on $G/\Gamma$,
it suffices to prove that for any $x\in G/\Gamma$
\[
\dim_H \{g\in B_{r}^G: gx \in \D(F^+, G/\Gamma)   \}\le d^-+d^c+d -\frac{\alpha_1}{\lambda}.
\]
where $r<1$  so that $B_{r}^G\subset B^-_1 B_{1/2}^cB_{1/2}^+$.
By Lemma
\ref{lem;stable},
\[
\{g\in B_{r}^G: gx \in \D(F^+, G/\Gamma)   \}\subset B^-_1\D_x^0,
\]
whose Hausdorff dimension is bounded from above by $\dim_H\D_x^0+d^-$.\footnote{Here we are using Marstrand's product theorem again.
	% which says that the Hausdorff dimension of the  product to two
%sets is bounded from above by the sum of the  Hausdorff dimension of one set and the packing dimension of the other.
}
In view of Theorem \ref{thm;simple}, the Hausdorff dimension of $\D(F^+, G/\Gamma)$
is less than $d+d^c+d^--\frac{\alpha_1}{\lambda}$ which is strictly less than
the manifold dimension of $G/\Gamma$.

\end{proof}

\begin{remark}
It is worth to mention that, if $F=A_F$,  then the contraction property of the Benoist-Quint height function proved
 in \cite{s} will allow us to prove a stronger result. Namely, we can get
a nontrivial upper bound of the Hausdorff dimension of the intersection of $\D(F^+, \ggm)$ and orbits of the so-called $(F^+, \Gamma)$-expanding subgroups introduced in \cite{KW}.
But unfortunately, we are not able to prove a uniform contracting property for the Benoist-Quint height functions even in the example  mentioned at the end of  the introduction due to the existence of the unipotent part in the second $\SL_4(\RR)$
factor.
\end{remark}
%The nice property of the height function in  \cite{em} is that it is
%constructed  using essentially only nontrivial irreducible  representations of the first $\SL_4(\RR)$ factor. The %advantage of the Benoit-Quint height function is
%that it allows the existence of nonzero invariant vectors of the first $\SL_4(\RR)$ factor, provided that the %invariant vectors are not changed by the dynamics, which in general  is not true in the setting of Theorem %\ref{thm;general}.


\begin{thebibliography}{99}

\bibitem{BQ}
Y. Benoist and J.-F. Quint, {\em Random walks on finite volume homogeneous spaces}, Invent. math. {\bf 187} (2012), 37--59.


\bibitem{bkm}
V. Bernik, D. Kleinbock and G.A. Margulis, {\em  Khintchine type theorems on
manifolds: The convergence case for standard multiplicative versions}, Internat.
Math. Res. Notices {\bf 9} (2001), 453--486.

\bibitem{borel}
A. Borel, {\em Linear algebraic groups. Second edition}, Graduate Texts in Mathematics, 126. Springer-Verlag, New York, 1991.

\bibitem{bp}
C. Bishop and Y.  Peres, {\em Fractals in probability and analysis},  Cambridge Studies in Advanced Mathematics 162,  Cambridge University Press, Cambridge, 2017.

\bibitem{c07}
Y. Cheung, {\em Hausdorff dimension of the set of points on divergent trajectories of a homogeneous flow on
a product space}, Ergodic Theory Dynam. Systems {\bf 27} (2007), 65--85.


\bibitem{c11}
Y. Cheung, {\em Hausdorff dimension of the set of singular pairs}, Ann. of Math. {\bf 173} (2011),
127--167.

\bibitem{cc16}
 Y. Cheung  and  N. Chevallier,  {\em  Hausdorff dimension of singular vectors},  Duke Math. J. {\bf 165} (2016), no. 12, 2273--2329.

 \bibitem{cg}
 L. Corwin and F. Greenleaf, {\em Representations of nilpotent Lie groups and their applica- tions, Part I: Basic theory and examples}, Cambridge Studies in Advanced Mathematics 18, Cambridge University Press, Cambridge, 1990.


\bibitem{d84}
S.G. Dani, {\em  On orbits of unipotent flows on homogeneous spaces},  Ergodic Theory Dynam. Systems {\bf 4} (1984), no. 1, 25--34.

\bibitem{d86}
S.G. Dani, {\em On orbits of unipotent flows on homogeneous spaces II},  Ergodic Theory Dynam. Systems {\bf 6} (1986), no. 2, 167--182.

\bibitem{Dani}
S.G. Dani, {\em Divergent trajectories of flows on homogeneous spaces and Diophantine approximation}, J. Reine Angew. Math. {\bf 359} (1985), 55--89.


\bibitem{dfsu}
T. Das, L. Fishman, D.  Simmons and  M.  Urbanski,
{\em A variational principle in the parametric geometry of numbers, with applications to metric Diophantine approximation},
C. R. Math. Acad. Sci. Paris {\bf 355} (2017), no. 8, 835--846.

\bibitem{ek12}
M. Einsiedler and S. Kadyrov,{\em  Entropy and escape of mass for $\SL_3(\RR)/\SL_3(\ZZ)$},
Israel Journal of Mathematics {\bf 190} (2012), 253--288.

\bibitem{elmv}
M. Einsiedler, E.  Lindenstrauss, P.   Michel  and A.   Venkatesh,  {\em The distribution of closed geodesics on the modular surface, and Duke's theorem}, Enseign. Math. (2) {\bf 58} (2012), no. 3, 249--313.

\bibitem{em}
A. Eskin, G.A. Margulis, , {\em Recurrence properties of random walks on finite volume homogeneous manifolds},  In: Random Walks and Geometry, pp. 431--444. de Gruiter, Berlin
(2004).

\bibitem{emm98}
A. Eskin, G.A. Margulis and S. Mozes, {\em Upper bounds and asymptotics in a
	quantitative version of the Oppenheim conjecture}, Ann. of Math. {\bf 147} (1998),
93--141.

\bibitem{k}
S. Kadyrov, {\em Entropy and escape of mass for Hilbert modular spaces},  J. Lie Theory {\bf 22} (2012), no. 3, 701--722.

\bibitem{KKLM}
S. Kadyrov, D. Kleinbock, E. Lindenstrauss and G.A. Margulis, {\em   Singular systems of linear forms and non-escape of mass in the space of lattices},  J. Anal. Math. {\bf 133} (2017), 253--277.

\bibitem{kp}
S. Kadyrov and A. Pohl, {\em Amount of failure of upper-semicontinuity of entropy in noncompact rank one situations, and Hausdorff dimension},
 Ergodic Theory Dynam. Systems {\bf 37} (2017), no. 2, 539--563.

\bibitem{KW}
D. Kleinbock and B. Weiss, {\em Modified Schmidt games and a conjectue of Margulis}, J. Mordern Dynamics {\bf 7} (2013), no. 3, 429--460.

\bibitem{lsst}
 L. Liao, R. Shi,  O.N. Solan and N. Tamam, {\em Hausdorff dimension of weighted singular vectors},  http://arxiv.org/abs/1605.01287.

\bibitem{Ma}
G.A. Margulis, {\em On the action of unipotent groups in the space of lattices, In Lie Groups and their representations}, Proc. of Summer School in Group Representations, Bolyai Janos Math. Soc., Akademai Kiado, Budapest, 1971, pp. 365--370, Halsted, New York, 1975.


\bibitem{mt94}
G.A. Margulis and  G.M. Tomanov, {\em Invariant measures for actions
	of unipotent groups over local fields on homogeneous spaces},
Invent. Math. {\bf 116} (1994), no. 1, 347--392.

\bibitem{mat}
P. Mattila, {\em  Geometry of sets and measures in Euclidean spaces. Fractals and rectifiability}, Cambridge Studies in Advanced Mathematics, 44. Cambridge University Press, Cambridge, 1995.

%\bibitem{sw}
%N. A. Shah, Limit distributions of expanding translates of certain orbits on homogeneous spaces, Proc. Indian Acad. Sci. Math. Sci. 106 (1996), 105-125.


\bibitem{s}
R. Shi, {\em  Pointwise equidistribution for one parameter diagonalizable group action on
homogeneous space}, http://arxiv.org/abs/1405.2067.

\bibitem{yang}
L. Yang, {\em Hausdorff dimension of divergent diagonal geodesics on product of finite--volume hyperbolic spaces}, Ergodic Theory and Dynamical Systems, to appear.

\bibitem{weiss}
B. Weiss, {\em Divergent trajectories on noncompact parameter spaces}, Geom. Funct. Anal., {\bf 14} (2004),
no. 1, 94--149.



\end{thebibliography}
\end{document}